\documentclass[12pt,a4paper]{amsart}
\usepackage{amsmath, amssymb, amsthm, eufrak, setspace, verbatim, array, color}
\usepackage[top=2.5cm, bottom=2cm, left=2cm, right=2cm]{geometry}
\usepackage[all]{xy}
\newtheorem{theorem}{Theorem}[section]
\newtheorem{lemma}[theorem]{Lemma}
\newtheorem{proposition}[theorem]{Proposition}

\newcommand{\suchthat}{\;\ifnum\currentgrouptype=16 \middle\fi|\;}
\newcommand{\Gal}[1]{\mbox{Gal}(#1)}
\newcommand{\Hol}[1]{\mbox{Hol}(#1)}
\newcommand{\Aut}[1]{\mbox{Aut}(#1)}
\newcommand{\Stab}[1]{\mbox{Stab}(#1)}
\newcommand{\Perm}[1]{\mbox{Perm}(#1)}
\newcommand{\Cl}[1]{\mbox{Cl}(#1)}
\newcommand{\divides}{ \hspace{0mm} \mid }
\newcommand{\ndivides}{ \hspace{0mm} \nmid }

\newcommand{\p}{ {\mathfrak p} } 
\renewcommand{\P}{ {\mathfrak P} } 
\newcommand{\q}{ {\mathfrak q} } 
 
\renewcommand{\OE}{ {\mathfrak O}_{E}}
\newcommand{\OF}{ {\mathfrak O}_{F}}
\newcommand{\OK}{ {\mathfrak O}_{K}}
\newcommand{\OL}{ {\mathfrak O}_{L}}

\newcommand{\OEp}{ {\mathfrak O}_{E,\p}}
\newcommand{\OEq}{ {\mathfrak O}_{E,\q}}
\newcommand{\OFq}{ {\mathfrak O}_{F,\q}}
\newcommand{\OFp}{ {\mathfrak O}_{F,\p}}
\newcommand{\OKp}{ {\mathfrak O}_{K,\p}}
\newcommand{\OLp}{ {\mathfrak O}_{L,\p}}
\newcommand{\OLP}{ {\mathfrak O}_{L,\P}}
\newcommand{\Tr}{ \mathrm{Tr} }
\newcommand{\A}{ {\mathfrak A}}
\newcommand{\Ap}{ {\mathfrak A}_{\p}}
\newcommand{\AH}{ {\mathfrak A}_{H}}
\newcommand{\M}{ {\mathfrak M} }
\newcommand{\B}{ {\mathfrak B} }
\renewcommand{\bar}{\overline}

\title{Hopf-Galois module structure of tamely ramified radical extensions of prime degree}
\author{Paul J. Truman}
\address{School of Computing and Mathematics, Keele University, Staffordshire, ST5 5BG, UK}
\email{P.J.Truman@Keele.ac.uk}


\begin{document}

\maketitle

\begin{abstract}
Let $ K $ be a number field and let $ L/K $ be a tamely ramified radical extension of prime degree $ p $. If $ K $ contains a primitive $ p\textsuperscript{th} $ root of unity then $ L/K $ is a cyclic Kummer extension; in this case the group algebra $ K[G] $ (with $ G=\Gal{L/K}) $ gives the unique Hopf-Galois structure on $ L/K $, the ring of algebraic integers $ \OL $ is locally free over $ \OK[G] $ by Noether's theorem, and G\'{o}mez Ayala has determined a criterion for $ \OL $ to be a free $ \OK[G] $-module. If $ K $ does not contain a primitive $ p\textsuperscript{th} $ root of unity then $ L/K $ is a separable, but non-normal, extension, which again admits a unique Hopf-Galois structure. Under the assumption that $ p $ is unramified in $ K $, we show that $ \OL $ is locally free over its associated order in this Hopf-Galois structure and determine a criterion for it to be free. We find that the conditions that appear in this criterion are identical to those appearing in G\'{o}mez Ayala's criterion for the normal case.
\end{abstract}

\section{Introduction} \label{section_introduction}

In classical Galois module theory we consider a finite Galois extension $ L/K $ of local or global fields with Galois group $ G $ and study the structure of each fractional ideal $ \B $ of $ L $ as a module over its associated order  $ \A_{K[G]}(\B) \subset K[G] $, with particular emphasis on the case $ \B=\OL $, the ring of algebraic integers (or valuation ring) of $ L $. Hopf-Galois module theory generalizes this situation: a {\em Hopf-Galois structure} on an extension of fields $ L/K $ consists of a $ K $-Hopf algebra $ H $ together with a certain $ K $-linear action of $ H $ on $ L $ (see \cite[Definition 2.7]{Childs00} for the precise definition). If $ L/K $ is an extension of local or global fields and $ H $ gives a Hopf-Galois structure on $ L/K $ then we can study the structure of each fractional ideal $ \B $ of $ L $ as a module over its associated order $ \AH(\B) \subset H $. These techniques have applications to Galois extensions $ L/K $: in this case the group algebra $ K[G] $  (with $ G=\Gal{L/K} $) gives a Hopf-Galois structure on $ L/K $, and any further Hopf-Galois structures admitted by the extension provide alternative contexts in which we can study each fractional ideal (see \cite{Byott97}, for example). However, the application of Hopf-Galois theory to extensions which are not Galois is particularly interesting, since it provides descriptions of rings of algebraic integers and/or fractional ideals in situations where classical techniques do not apply. For example, Hopf-Galois theory has recently been used by Koch \cite{Koch15}  to study the structure of fractional ideals in a totally ramified purely inseparable extension of local fields of prime power degree, and by Elder \cite{Elder18} to address the same questions for a separable, but non-normal, ramified extension of local fields of prime degree. In this paper we study the Hopf-Galois module structure of the ring of algebraic integers in a tamely ramified non-normal radical extension of number fields of prime degree. These are the first results concerning the Hopf-Galois module structure of rings of algebraic integers in non-normal extensions of global fields. 
\\ \\
Let $ K $ be a number field, $ p $ a prime number, and $ \zeta \in \mathbb{C} $ a primitive $ p\textsuperscript{th} $ root of unity. Let $ L=K(\omega) $ with $ \omega^{p} \in K - K^{p} $. If $ \zeta \in K $ then $ L/K $ is a cyclic Kummer extension, and the group algebra $ K[G] $ (with $ G=\Gal{L/K} $) gives a Hopf-Galois structure on $ L/K $. If $ L/K $ is at most tamely ramified (henceforth, ``tame") then Noether's theorem \cite[Theorem 3]{Frohlich83} implies that $ \A_{K[G]}(\OL)=\OK[G] $ and that $ \OL $ is a locally free $ \OK[G] $-module (of rank one), and G\'{o}mez Ayala \cite{GomezAyala94} has determined a criterion for $ \OL $ to be a free $ \OK[G] $-module. By Byott's uniqueness theorem \cite[Theorem 1]{Byott96} the Hopf-Galois structure given by $ K[G] $ is the only Hopf-Galois structure admitted by the extension, and so we have nothing to add to these results. If $ \zeta \not \in K $ then $ L/K $ is a separable, but non-normal, extension. Results of Childs \cite[Section 2]{Childs89} and Kohl \cite[Theorem 3.3]{Kohl98} imply that $ L/K $ also admits a unique Hopf-Galois structure in this case. We show that if $ L/K $ is tame and $ p $ is unramified in $ K $ then $ \OL $ is a locally free module (of rank one) over its associated order in this Hopf-Galois structure and determine a criterion for $ \OL $ to be free. Interestingly, we find that the conditions that appear in this criterion are identical to those that appear in G\'{o}mez Ayala's criterion for the normal case. 
\\ \\
The paper is organized as follows. In section \ref{section_classical_version} we summarize G\'{o}mez Ayala's work on the case in which $ L/K $ is normal; thereafter we assume that $ L/K $ is non-normal. From section \ref{section_tame_radical_extensions} onward we assume that $ p $ is unramified in $ K $; under this assumption we establish a criterion for $ L/K $ to be tame (proposition \ref{prop_criterion_for_tame}), and determine an integral basis of $ \OLp = \OKp \otimes_{\OK} \OL $ over $ \OKp $ for each prime ideal $ \p $ of $ \OK $ in this case (propositions \ref{prop_integral_basis_away_from_p} and \ref{prop_integral_basis_above_p}). In section \ref{section_HGS} we study the unique Hopf-Galois structure admitted by $ L/K $; in particular, we show that the Hopf algebra $ H $ giving this Hopf-Galois structure is isomorphic to $ K^{p} $ as a $ K $-algebra (proposition \ref{prop_K_algebra_isomorphism}) and give a simple formula for its action on $ L $ (proposition \ref{prop_action_of_idempotents}).  In section \ref{section_HGMS} we show that $ \OL $ is locally free over its associated order $ \A = \AH(\OL) $ (theorem \ref{thm_local_generators}), and for each $ \p $ we determine an explicit generator of $ \OLp $ over $ \Ap $. Given that $ \OL $ is a locally free $ \A $-module, a result of Bley and Johnston \cite[Proposition 2.1]{BleyJohnston08} relates the structure of $ \OL $ as an $ \A $-module to the structure of $ \M\OL $ as an $ \M $-module, where $ \M $ denotes the unique maximal order in $ H $; we use this result, along with an id\'{e}lic description of the locally free class group $ \Cl{\M} $, to derive a criterion for  for $ \OL $ to be a free $ \A $-module (theorem \ref{thm_criterion_free}), and show that the conditions appearing in this criterion are identical to those appearing in G\'{o}mez Ayala's criterion for the normal case. Finally, in section \ref{section_uniform} we discuss a unified approach to the normal and non-normal cases. 

\section{G\'{o}mez Ayala's Criterion} \label{section_classical_version}

We retain the notation established in the introduction: $ K $ is a number field, $ p $ a prime number, $ \zeta $ a primitive $ p\textsuperscript{th} $ root of unity, and $ L/K $ is an extension of the form $ L = K(\omega) $ with $ \omega^{p} \in K - K^{p} $. In this section we assume that $ \zeta \in K $; the extension $ L/K $ is then a cyclic Kummer extension and the group algebra $ K[G] $ (with $ G = \Gal{L/K} $) gives the unique Hopf-Galois structure on $ L/K $. We also suppose that $ L/K $ is tame. By Noether's theorem we then have that $ \A_{K[G]}(\OL) = \OK[G] $ and $ \OL $ is a locally free $ \OK[G] $-module.  G\'{o}mez Ayala \cite{GomezAyala94} has determined a criterion for $ \OL $ to be a free $ \OK[G] $-module; we summarize his result using some notation and terminology from \cite{CorsoRossi10}. Each ideal $ \mathfrak{a} $ of $ \OK $ has a unique decomposition of the form
\[ \mathfrak{a} =  \prod_{i \geq 1} \mathfrak{a}_{i}^{i}, \]
where the $ \mathfrak{a}_{i} $ are pairwise coprime squarefree ideals of $ \OF $. (We have $ \mathfrak{a}_{i} = \OK $ for $ i $ sufficiently large, so the product above is finite.) We call the ideal $ \mathfrak{a}_{i} $ the {\em $ i $-part of $ \mathfrak{a} $}, and note that it is the product of the prime ideals $ \p $ of $ \OK $ such that $ v_{\p}(\mathfrak{a}) =i $. Next we define the {\em ideals associated to $ \mathfrak{a} $} by
\[ \mathfrak{b}_{j} = \prod_{i \geq 1} \mathfrak{a}_{i}^{\lfloor ij / p \rfloor} \mbox{ for $ 0 \leq j \leq p-1 $},\]
where $ \lfloor x \rfloor $ denotes the largest integer not exceeding $ x $. We remark that an alternative expression for the $ \mathfrak{b}_{j} $ is
\[ \mathfrak{b}_{j} = \prod_{\p} \p^{\lfloor v_{\p}(\mathfrak{a}^{j}) / p \rfloor} = \prod_{\p} \p^{r_{\p}(\mathfrak{a}^{j})} \mbox{ for $ 0\leq j \leq p-1 $}, \]
where the product is taken over the prime ideals $ \p $ of $ \OK $, and $ r_{\p}(\mathfrak{a}^{j}) = \lfloor v_{\p}(\mathfrak{a}^{j}) / \p \rfloor $.
\\ \\
G\'{o}mez Ayala's result is that $ \OL $ is a free $ \OK[G] $-module if and only if there exists an element $ \beta \in \OL $ such that 
\begin{enumerate}
\item $ L=K(\beta) $, 
\item $ b=\beta^{p} \in \OK $,
\item the ideals  $ \mathfrak{b}_{j} $ associated to $ b\OK $ are principal with generators $ b_{j} $ such that 
\[ \sum_{j=0}^{p-1} \frac{\beta^{j}}{b_{j}} \equiv 0 \pmod{p \OL}. \]
\end{enumerate}
Furthermore, in this case the element
\[ \frac{1}{p} \sum_{j=0}^{p-1} \frac{\beta^{j}}{b_{j}}  \]
is a free generator of $ \OL $ as an $ \OK[G] $-module. 
\\ \\
Several authors have studied generalizations or variants of this result. Ichimura \cite{Ichimura09} proved that if $ p $ is unramified in $ K $ and $ L/K $ is a tamely ramified Galois extension of degree $ p $ then $ L $ has a normal integral basis if the Kummer extension $ L(\zeta)/K(\zeta) $ has a normal integral basis. Ichimura also studied the case in which $ L/K $ is a cyclic Kummer extension of arbitrary degree \cite{Ichimura04}, and a criterion for the existence of a normal integral basis in this case was given by Del Corso and Rossi \cite{CorsoRossi10}. 

\section{Tame radical extensions of prime degree} \label{section_tame_radical_extensions}

Henceforth we suppose that $ \zeta \not \in K $, and write $ F = K(\zeta) $. The extension $ L/K $ is then separable but non-normal. In fact, we impose the stronger hypothesis that $ p $ is unramified in $ K $. We record some consequences of this assumption:  

\begin{lemma} \label{lem_consequences_of_p_unramified_in_K}
Suppose that $ p $ is unramified in $ K $. Then:
\begin{enumerate}
\item $ F/K $ has degree $ p-1 $;
\item each prime ideal $ \p $ of $ \OK $ lying above $ p $ is totally ramified in $ F/K $;
\item the set $  \{1,\zeta, \ldots ,\zeta^{p-2} \} $ is an integral basis of $ \OF $ over $ \OK $. 
\end{enumerate}
\end{lemma}
\begin{proof}
Let $ \p $ be a prime ideal of $ K $ lying above $ p $. Since $ p $ in unramified in $ K $ the polynomial $ f(x)=x^{p-1}+x^{p-2}+ \cdots + x + 1 $ is an Eisenstein polynomial over $ K_{\p} $ and has $ \zeta $ as a root, so $ K_{\p}(\zeta) / K_{\p} $ has degree $ p-1 $ and is totally ramified, by \cite[Theorem 24]{FrohlichTaylor93}. Therefore $ F/K $ has degree $ p-1 $ and $ \p $ is totally ramified in $ F/K $. We also have from \cite[Theorem 24]{FrohlichTaylor93} that the set $ \{1,\zeta, \ldots ,\zeta^{p-2} \}  $ is an $ \OKp $-basis of $ \OFp $. Since $ \mathfrak{d}(\{1,\zeta, \ldots ,\zeta^{p-2} \} ) = \pm p^{p-2} $, this set is actually an integral basis of $ \OFp $ over $ \OKp $ for all prime ideals $ \p $ of $ \OK $, and therefore an integral basis of $ \OF $ over $ \OK $. 
\end{proof}

The Galois closure of $ L/K $ is $ E=K(\zeta, \omega) $, and $ E/F $ is a Galois extension of degree $ p $. In this section we use well known results concerning ramification and local integral bases in $ E/F $ to establish a criterion for $ L/K $ to be tame and to determine local integral bases in this case. 

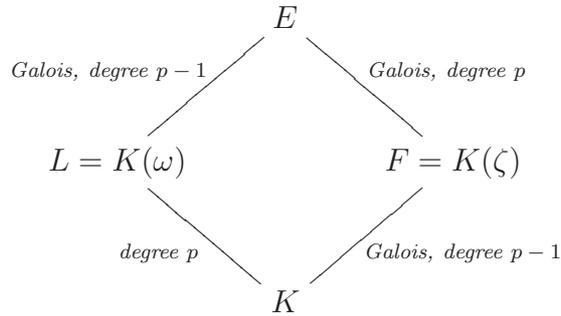
\[
\xymatrixcolsep{2pc} 
\xymatrixrowsep{3pc}
\xymatrix{
& \ar@{-}[dl]_{\textit{Galois, degree $p-1$}} \ar@{-}[dr]^{\textit{Galois, degree $p$}} E & \\
L=K(\omega) \ar@{-}[dr]_{\textit{degree $p$}}& & F=K(\zeta) \ar@{-}[dl]^{\textit{Galois, degree $p-1$}} \\
& K & } 
\]

\begin{proposition} \label{prop_translate_tameness}
$ L/K $ is tame if and only if $ E / F $ is tame.
\end{proposition}
\begin{proof}
The extension $ F/K $ is tame since it is Galois of degree $ p-1 $ and is ramified only at prime ideal lying above $ p $; similarly $ E/L $ is tame. If $ L/K $ is tame then since $ E/L $ is tame we have that $ E/K $ is tame, and so $ E/F $ is tame. Conversely, if $ E/F $ is tame then since $ F/K $ is tame we have that $ E/K $ is tame, and so $ L/K $ is tame. 
\end{proof}

\begin{proposition} \label{prop_criterion_for_tame}
$ L/K $ is tame if and only if there exists $ \alpha \in \OL $ such that
\begin{enumerate}
\item $ L=K(\alpha) $;
\item $ \alpha^{p} \equiv 1 \pmod{p^{2}\OK} $.
\end{enumerate}
\end{proposition}
\begin{proof}
Suppose first that there exists $ \alpha \in \OL $ with the properties stated in the proposition. Then (since $ (\zeta-1)^{p-1}\OF=p\OF $) we have $ E = F(\alpha) $ with $ \alpha^{p} \equiv 1 \pmod{(\zeta-1)^{p}\OF} $ and so by \cite[Propositions (24.2) and (24.4)]{Childs00} each prime ideal $ \q $ of $ \OF $ lying above $ p $ is unramified in $ E $. Since $ E/F $ is a Galois extension of degree $ p $ this implies that $ E/F $ is tame, and so $ L/K $ is tame by proposition \ref{prop_translate_tameness}.

Conversely, suppose that $ L/K $ is tame. Then by proposition \ref{prop_translate_tameness} $ E/F $ is tame, and so by \cite[Propositions (24.2) and (24.4)]{Childs00}  for each prime ideal $ \q $ of $ \OF $ lying above $ p $ there exists $ \beta_{\q} \in \OEq $ such that $ E_{\q} = F_{\q}(\beta_{\q}) $ and $ \beta_{\q}^{p} \equiv 1 \pmod{(\zeta-1)^{p}\OFq} $. By the Chinese Remainder Theorem there exists $ \beta \in \OE $ such that $ E = F(\beta) $ and $ \beta^{p} \equiv 1 \pmod{(\zeta-1)^{p}\OF} $. Let $ \alpha = N_{E/L}(\beta) $; then $ L=K(\alpha) $ and $ \alpha^{p} \equiv 1 \pmod{(\zeta-1)^{p}\OF} $. But $ \alpha^{p} \in \OF \cap \OL = \OK $, so (again using the fact that $ (\zeta-1)^{p-1}\OF=p\OF  $) we have $ \alpha^{p} \equiv 1 \pmod{p^{2}\OK} $.

\end{proof}

Henceforth we shall suppose that $ L/K $ is tame and that $ L=K(\alpha) $ with $ a = \alpha^{p} \equiv 1 \pmod{p^{2}\OK} $. We now determine integral bases of $ \OLp $ over $ \OKp $ for each prime ideal $ \p $ of $ \OK $. 

\begin{proposition} \label{prop_integral_basis_away_from_p}
Let $ \p $ be a prime ideal of $ \OK $ that does not lie above $ p $, and let $ \pi_{\p} $ be a uniformizer of $ K_{\p} $. For $ x \in K $ let $ {\displaystyle  r_{\p}(x) = \left\lfloor \frac{v_{\p}(x)}{p} \right\rfloor } $. Then an integral basis of $ \OLp $ over $ \OKp $ is given by 
\[ \left\{ \frac{ \alpha^{j} }{ \pi_{\p}^{r_{\p}(a^{j})}} \suchthat j=0,1, \ldots ,p-1 \right\}. \]
\end{proposition}
\begin{proof}
The ramification indices of $ \p $ in $ F/K $ and $ E/K $ depend only on $ \p $, because $ F/K $ and $ E/K $ are Galois extensions. Let $ \q $ be a prime ideal of $ \OF $ that lies above $ \p $. Then we have $ e_{\p}(E/K) = e_{\q}(E/F)e_{\p}(F/K)  $, but $ e_{\p}(F/K) = 1 $ since $ \p $ does not lie above $ p $, and $ e_{\q}(E/F) = 1 $ or $ p $, since $ E/F $ is a Galois extension of degree $ p $. Hence $ e_{\p}(E/K) = 1 $ or $ p $. Now let $ \P $ be a prime ideal of $ \OL $ lying above $ \p $. Then $ \P $ is unramified in $ E/L $, since $ E=L(\zeta) $ and $ \P $ does not lie above $ p $. Therefore $ \p $ is either unramified or totally ramified in $ L/K $, according to whether $ \q $ is unramified or totally ramified in $ E/F $. By \cite[Theorem 118]{Hecke81}, $ \q $ is unramified in $ E/F $ if $ p \divides v_{\q}(a) $, and totally ramified if $ p \ndivides v_{\q}(a) $. Since $ \p $ is unramified in $ F/K $ we conclude that $ \p $ is unramified in $ L/K $ if $ p \divides v_{\p}(a) $, and totally ramified if $ p \ndivides v_{\p}(a) $. 
\\ \\
If $ p \divides v_{\p}(a) $ then for each prime ideal $ \P $ of $ \OL $ lying above $ \p $  and each $ j=0,1, \ldots ,p-1 $  we have $ r_{\p}(a^{j}) = v_{\P}(\alpha^{j}) $, so $ v_{\P} \left(  \alpha^{j} / \pi_{\p}^{r_{\p}(a^{j})} \right) = 0 $ (since $ v_{\P}(\pi_{\p})=1 $). Therefore the discriminant of the set in the proposition lies in $ \OKp^{\times} $, and so this set is an $ \OKp $-basis of $ \OLp $. 
\\ \\
If $ p \ndivides v_{\p}(a) $ then let $ \P $ be the unique prime ideal of $ \OL $ that lies above $ \p $;  we then have $ v_{\P}(\pi_{\p})=p $ and $ v_{\P}(\alpha)=v_{\p}(a) $, so for each $ j=0,1, \ldots ,p-1 $
\[ v_{\P} \left(  \frac{\alpha^{j}} {\pi_{\p}^{r_{\p}(a^{j})}} \right) = v_{\p}(a^{j}) - pr_{\p}(a^{j}). \]
Therefore for each $ j=0,1, \ldots ,p-1 $ the set in the proposition contains exactly one element whose valuation at $ \P $ is equal to $ j $, and so this set is an $ \OKp $-basis of $ \OLP = \OLp $. 
\end{proof}

\begin{proposition} \label{prop_integral_basis_above_p}
Let $ \p $ be a prime ideal of $ \OK $ that lies above $ p $. Then an integral basis of $ \OLp $ over $ \OKp $ is given by
\[ \left\{ 1, \alpha, \ldots ,\alpha^{p-2}, \frac{1}{p}\left(1+\alpha+ \cdots +\alpha^{p-1} \right) \right\}. \]
\end{proposition}
\begin{proof}
Let $ \q $ be the unique prime ideal of $ \OF $ that lies above $ \p $. By \cite[Proposition 24.4]{Childs00} an $ \OFq $-basis of $ \mathfrak{O}_{E,\q} $ is given by
\[ \left\{ \left( \frac{ \alpha-1 }{ \zeta -1 } \right)^{j} \suchthat j=0,1, \ldots ,p-1 \right\}. \]
Since $ v_{\q}(p)=p-1 $, it follows that an $ \OFq $-basis of $ \mathfrak{O}_{E,\q} $ is given by
\[ \left\{ \frac{(\alpha-1)^{j}}{p} (\zeta-1)^{p-1-j} \suchthat j=0,1, \ldots ,p-1 \right\} \]
and (using lemma \ref{lem_consequences_of_p_unramified_in_K}) that an $ \OKp $-basis of $ \mathfrak{O}_{E,\p} $ is given by 
\[ \left\{  \frac{(\alpha-1)^{j}}{p} \zeta^{i}(\zeta-1)^{p-1-j}  \suchthat \begin{array}{l} i=0,1, \ldots ,p-2 \\ j=0,1, \ldots ,p-1 \end{array} \right\}. \]
Since $ E/L $ is tamely ramified, we have $ \OL = \Tr_{E/L}(\OE) $, and so $ \OLp $ is spanned over $ \OKp $ by the images of the elements of this set under the map $ \Tr = \Tr_{E/L} $. Since $ [F:K]=p-1 $ by lemma \ref{lem_consequences_of_p_unramified_in_K}, we have
\[ \Tr(\zeta^{r}) = \left\{ \begin{array}{cl} p-1 & \mbox{if } r \equiv 0 \pmod{p} \\ -1 & \mbox{otherwise.} \end{array} \right. \] 
Now for $ i,j=0,1, \ldots p-2 $ we have
\begin{eqnarray*}
\Tr \left( \frac{(\alpha-1)^{j}}{p} \zeta^{i}(1-\zeta)^{p-1-j} \right) & = & \frac{(\alpha-1)^{j}}{p} \sum_{k=0}^{p-1-j} \binom{p-1-j}{k} \Tr\left( \zeta^{i}(-\zeta)^{k} \right) \\
& = &  \frac{(\alpha-1)^{j}}{p} \left( (-1)^{p-i}  \binom{p-1-j}{p-i}p - \sum_{k=0}^{p-1-j} \binom{p-1-j}{k} (-1)^{k} \right) \\ 
& = &  (-1)^{p-i}  \binom{p-1-j}{p-i}(\alpha-1)^{j} ,
\end{eqnarray*}
where we interpret the binomial coefficient $ \binom{p-1-j}{p-i} $ as zero if $ p-i > p-1-j $. The $ \OKp $-span of these elements is equal to the $ \OKp $-span of $ 1,\alpha, \ldots ,\alpha^{p-2} $. For $ i=0,1, \ldots ,p-2 $ and $ j=p-1 $ we have
\[ \Tr \left( \frac{(\alpha-1)^{p-1}}{p} \zeta^{i} \right) = \left\{ \begin{array}{ll} (\alpha-1)^{p-1} - \frac{(\alpha-1)^{p-1}}{p}& \mbox{if $ i=0 $} \\ - \frac{(\alpha-1)^{p-1}}{p} & \mbox{otherwise.} \end{array} \right. \]
Therefore $ \frac{(\alpha-1)^{p-1}}{p} \in \OLp $ and, using the fact that $ \binom{p-1}{k} \equiv (-1)^{k} \pmod{p} $, we have \linebreak $ \frac{1}{p}\left(1+\alpha+ \cdots +\alpha^{p-1} \right) \in \OLp $. Therefore the set in the proposition spans $ \OLp $ over $ \OKp $. Since this set is clearly linearly independent over $ \OKp $, it  forms an $ \OKp $-basis of $ \OLp $.
\end{proof}

\section{The Hopf-Galois structure on a  radical extension of prime degree} \label{section_HGS}

As discussed in the introduction, the extension $ L/K $ admits a unique Hopf-Galois structure. This fact is established in, for example, \cite[Section 2]{Childs89} or \cite[Theorem 3.3]{Kohl98}, but we give a self contained proof in our case for the convenience of the reader. We also establish some properties of this Hopf-Galois structure, which will be useful in what follows. 
\\ \\
We rewrite the Galois closure of $ L/K $ as $ E= K(\alpha,\zeta) $; we then have $ \Gal{E/K} = \langle \sigma, \tau \rangle $, where $ \sigma(\alpha) = \zeta \alpha $, $ \sigma(\zeta)=\zeta $, $ \tau(\alpha) = \alpha $, and $ \tau(\zeta) = \zeta^{d} $ for some primitive root $ d $ modulo $ p $. We have $ \sigma^{p} = \tau^{p-1}=1 $ and $ \tau \sigma \tau^{-1} = \sigma^{d} $. Let $ G = \Gal{E/K} $, $ G^{\prime}=\langle \tau \rangle $, and let $ X $ denote the left coset space $ G / G^{\prime} $. By a theorem of Greither and Pareigis (\cite[Theorem 2.1]{GreitherPareigis87} or \cite[Theorem 6.8]{Childs00}), the Hopf-Galois structures on $ L/K $ are in bijective correspondence with regular subgroups of $ \Perm{X} $ that are normalized by the image of $ G $ under the left translation map $ \lambda : G \rightarrow \Perm{X} $. Since $ |X|=p $, any such subgroup must be cyclic of order $ p $. We shall reformulate the problem via Byott's translation  theorem (\cite[Proposition 1]{Byott96} or \cite[Theorem 7.3]{Childs00}): let $ M = \langle \mu \rangle  $ be an abstract group of order $ p $, and recall that the {\em holomorph} of $ M $ is the group $ \Hol{M} \cong M \rtimes \Aut{M} $;  the appropriate subgroups of $ \Perm{X} $ are then in bijective correspondence with equivalence classes of embeddings $ \beta: G \hookrightarrow \Hol{M} $ such that $ \beta(G^{\prime} ) = \Stab{1_{M}} $, modulo conjugation by elements of $ \Aut{M} $. 

\begin{proposition} \label{prop_unique_HGS}
The extension $ L/K $ admits exactly one Hopf-Galois structure. 
\end{proposition}
\begin{proof}
Let $ \theta \in \Aut{M} $ be defined by $ \theta(\mu) = \mu^{d} $; it is then easy to see that $ \beta: G \rightarrow \Hol{M} $ defined by
\[  \beta(\sigma) = (\mu,1) \mbox{ and } \beta(\tau) = (1,\theta) \]
is an embedding $ \beta: G \hookrightarrow \Hol{M} $ such that $ \beta(G^{\prime}) = \Stab{1_{M}} $. If $ \beta^{\prime} $ is another such embedding then since $ (M,1) $ is the unique Sylow $ p $-subgroup of $ \Hol{M} $ we have $ \beta^{\prime}(\sigma) = \mu^{i} $ for some integer $ i $. Similarly, since $ \beta^{\prime}(\tau)[1_{M}] = 1_{M} $ and $ \tau $ has order $ p-1 $ we have $ \beta^{\prime}(\tau)=(1,\theta^{j}) $ for some integer $ j $ coprime to $ p-1 $. Now we have
\[ \beta^{\prime}(\tau \sigma \tau^{-1}) = \beta^{\prime}(\sigma^{d}) = (\mu^{id},1), \]
but also
\[ \beta^{\prime}(\tau \sigma \tau^{-1})  = (1,\theta^{j})(\mu^{i},1)(1,\theta^{-j}) = (\mu^{ijd},1). \]
Hence $ j=1 $. Now let $ \varphi \in \Aut{M} $ be defined by $ \varphi(\mu)=\mu^{i} $. Then
\[ \beta^{\prime}(\sigma) = (\mu^{i},1) = \varphi (\mu,1) \varphi^{-1} = \varphi \beta(\sigma) \varphi^{-1}, \]
and (since $ \Aut{M} $ is abelian)
\[ \beta^{\prime}(\tau) = (1,\theta) =  \varphi (1,\theta) \varphi^{-1} = \varphi \beta(\tau) \varphi^{-1}. \]
Therefore $ \beta^{\prime} $ and $ \beta $ are conjugate by an element of $ \Aut{M} $, so there is exactly one equivalence class of suitable embeddings $ \beta: G \hookrightarrow \Hol{M} $, and so exactly one Hopf-Galois structure on $ L/K $. 
\end{proof}

By using elements of the proof of Byott's translation theorem, we can determine the regular subgroup of $ \Perm{X} $ that corresponds to the unique Hopf-Galois structure on $ L/K $: 

\begin{proposition} \label{prop_determining_regular_subgroup}
Let $ \eta \in \Perm{ X } $ be defined by $ \eta(\bar{\sigma}^{i}) = \bar{\sigma}^{i-1} $. Then $ N = \langle \eta \rangle $ is the regular subgroup of $ \Perm{X} $ that corresponds to the unique Hopf-Galois structure on $ L/K $. 
\end{proposition}
\begin{proof}
Let $ \beta: G \rightarrow \Hol{M} $ be defined by
\[  \beta(\sigma) = (\mu,1) \mbox{ and } \beta(\tau) = (1,\theta), \]
as in proposition \ref{prop_unique_HGS}. From $ \beta $ we obtain a bijection $ b : X \rightarrow M $ defined by $ b(\bar{\sigma}^{i} ) = \beta(\sigma^{i})[1_{M}] = \mu^{i} $, where $ \bar{\sigma}^{i} = \sigma^{i}G^{\prime} $. The map $ \widehat{\beta} : M \rightarrow \Perm{X} $ defined by $  \widehat{\beta}(\mu) = b^{-1} \lambda_{M}(\mu) b $ (where $ \lambda_{M} $ denotes the left regular representation of $ M $) is then an embedding of $ M $ into $ \Perm{X} $ whose image is regular and normalized by $ \lambda(G) $, and $  \widehat{\beta}(M) $ is the regular subgroup of $ \Perm{X} $ that corresponds to $ \beta $. We have:
\begin{eqnarray*}
 \widehat{\beta}(\mu)[\bar{\sigma}^{i}] & = & b^{-1} [\lambda_{M}(\mu) b[\bar{\sigma}^{i}]] \\
& = & b^{-1}[ \lambda_{M}(\mu) \mu^{i}] \\
& = & b^{-1} [\mu^{i+1}] \\
& = & \bar{\sigma}^{i+1} \\
& = & \eta^{-1}(\bar{\sigma}^{i}),
\end{eqnarray*}
and so $  \widehat{\beta}(M) = N  $.
\end{proof}

The theorem of Greither and Pareigis also asserts that the Hopf algebra giving the Hopf-Galois structure corresponding to $ N $ is $ H=E[N]^{G} $, where $ G $ acts on $ E $ as Galois automorphisms and on $ N $ by conjugation via $ \lambda $, viz. $ \,^{g} \eta = \lambda(g) \eta \lambda(g)^{-1} $ for all $ g \in G $. 

\begin{proposition} \label{prop_K_algebra_isomorphism}
We have $ H \cong K^{p} $ as $ K $-algebras.
\end{proposition}
\begin{proof}
Since $ \zeta \in E $ the group algebra $ E[N] $ has an $ E $-basis of mutually orthogonal idempotents
\[ e_{i} = \frac{1}{p} \sum_{k=0}^{p-1} \zeta^{-ik} \eta^{k} \mbox{ for $ i=0,1, \ldots ,p-1 $},\]
and so $ E[N] \cong E^{p} $ as $ E $-algebras. It is easy to verify that $ \,^{\sigma} \eta = \eta $ and $ \,^{\tau} \eta = \eta^{d} $; it follows that each idempotent $ e_{i} $ is fixed by each element of $ G $, and so lies in $ E[N]^{G} = H $. Therefore $ H $ has a $ K $-basis of mutually orthogonal idempotents, and so $ H \cong K^{p} $ as $ K $-algebras. 
\end{proof}

Finally, the theorem of Greither and Pareigis implies that the action of $ H $ on $ L $ is given by
\[ \left( \sum_{k=0}^{p-1} c_{k} \eta^{k} \right) \cdot x = \sum_{k=0}^{p-1} c_{k} \eta^{-k}[\bar{1_{G}}](x) = \sum_{k=0}^{p-1} c_{k} \bar{\sigma}^{k} (x) \mbox{ for all $ x \in L $.}\]

\begin{proposition} \label{prop_action_of_idempotents}
For $ i,j = 0,1, \ldots , p-1 $ we have $ e_{i} \cdot \omega^{j} = \delta_{i,j} \omega^{j} $.
\end{proposition}
\begin{proof}
For $ i,j=0,1, \ldots ,p-1 $ we have
\begin{eqnarray*}
e_{i} \cdot \omega^{j} & = & \left( \frac{1}{p} \sum_{k=0}^{p-1} \zeta^{-ik} \eta^{k} \right) \cdot \omega^{j} \\
& = &  \frac{1}{p} \sum_{k=0}^{p-1} \zeta^{-ik} \bar{\sigma}^{k} \omega^{j} \\
& = &  \frac{1}{p} \sum_{k=0}^{p-1} \zeta^{-ik} \zeta^{jk} \omega^{j} \\
& = &  \frac{1}{p} \sum_{k=0}^{p-1} \zeta^{k(j-i)} \omega^{j} \\
& = & \delta_{i,j} \omega^{j}. 
\end{eqnarray*} 
\end{proof}

\section{Hopf-Galois module structure} \label{section_HGMS}

In this section we show that $ \OL $ is locally free over its associated order $ \A $ in $ H $, and determine a criterion for it to be free. We have previously studied the Hopf-Galois module structure of fractional ideals in tame extensions of local or global fields (see \cite{Truman11}, \cite{Truman13}, or \cite{Truman18}, for example), and some of our results could be applied here: for example \cite[Proposition 5.6]{Truman11} implies that $ \OLp $ is a free $ \Ap $-module for each prime ideal $ \p $ of $ \OK $ that does not lie above $ p $. However, our existing results do not apply to the prime ideals lying above $ p $ in our current situation, or construct explicit generators, which we shall require in what follows. Therefore the following proposition is necessary:

\begin{theorem} \label{thm_local_generators}
We have $ \A=\OE[N]^{G} $, and $ \OL $ is a locally free $ \A $-module. For each prime ideal $ \p $ of $ \OK $, a free generator of $ \OLp $ as an $ \Ap $-module is given by
\[ x_{\p} = \left\{ \begin{array}{ >{\displaystyle}ll}
\frac{1}{p}\sum_{j=0}^{p-1} \alpha^{j} & \mbox{ if $ \p \divides p\OK $} \\
 \frac{1}{p} \sum_{j=0}^{p-1} \frac{ \alpha^{j}}{\pi_{\p}^{r_{\p}(a^{j})}} & \mbox{otherwise.}
 \end{array} \right. \]
\end{theorem}
\begin{proof}
By \cite[Proposition 2.5]{Truman11} we have $ \OE[N]^{G} \subseteq \A $. On the other hand, note that for $ i=0,1, \ldots ,p-1 $ we have $ pe_{i} \in \OE[N]^{G} $. We shall show that for each prime ideal $ \p $ of $ \OK $ the set $ \{ x_{\p}, pe_{1}\cdot x_{\p}, \ldots ,pe_{p-1} \cdot x_{\p} \} $ (with $ x_{\p} $ as defined in the proposition) is an $ \OKp $-basis of $ \OLp $. This will imply that $ \OL $ is a locally free $ \OE[N]^{G} $-module, and hence that $ \A = \OE[N]^{G} $. Recall from proposition \ref{prop_action_of_idempotents} that for $ i,j = 0,1, \ldots , p-1 $ we have $ e_{i} \cdot \alpha^{j} = \delta_{i,j} \alpha^{j} $. If $ \p $ lies above $ p $ then we have 
\begin{eqnarray*}
1 \cdot x_{\p} & = & x_{\p} \\
pe_{i} \cdot x_{\p} & = & \alpha^{i} \mbox{ for $ i=1,2, \ldots ,p-1 $ }.
\end{eqnarray*}
Referring to proposition \ref{prop_integral_basis_above_p} we see that $ \OLp $ is a free $ \OEp[N]^{G} $-module with generator $ x_{\p} $. If $ \p $ does not lie above $ p $ then we have
\begin{eqnarray*}
1 \cdot x_{\p} & = & x_{\p} \\
pe_{i} \cdot x_{\p} & = & \frac{ \alpha^{i}}{\pi_{\p}^{r_{\p}(a^{i})}} \mbox{ for $ i=1,2, \ldots ,p-1 $ }.
\end{eqnarray*}
Referring to proposition \ref{prop_integral_basis_away_from_p} and recalling that $ p \in \OKp^{\times} $ in this case, we see as above that $ \OLp $ is a free $ \OEp[N]^{G} $-module with generator $ x_{\p} $. Therefore $ \OL $ is a locally free $ \OE[N]^{G} $-module, so $ \A = \OE[N]^{G} $. 
\end{proof}

To establish a criterion for $ \OL $ to be a free $ \A $-module, we use a result of Bley and Johnston \cite[Proposition 2.1]{BleyJohnston08}. The $ K $-algebra $ H $ contains a unique maximal $ \OK $-order $\M $, which is the  preimage of $ \OK^{p} $ under the isomorphism $ H \cong K^{p} $ constructed in proposition \ref{prop_K_algebra_isomorphism}. Let $ \M\OL $ denote the smallest $ \M $-submodule of $ L $ that contains $ \OL $. Then the result of Bley and Johnston implies that $ \OL $ is a free $ \A $-module if and only if

\begin{itemize}
\item $ \OL $ is a locally free $ \A $-module;
\item $ \M \OL $ is a free $ \M $-module, and $ \M \OL = \M \cdot x $ for some $ x \in \OL $. 
\end{itemize}

Furthermore, in this case the element $ x \in \OL $ is a free generator of $ \OL $ as an $ \A $-module. We have shown in theorem \ref{thm_local_generators} that $ \OL $ is a locally free $ \A $-module, and so we now focus our attention on the $ \M $-module $ \M\OL $. We shall first establish a criterion for $ \M\OL $ to be a free $ \M $-module, and then turn to the question of when it has a free generator lying in $ \OL $. Certainly  
 $ \M \OL $ is a locally free $ \M $-module, and $ (\M \OL)_{\p} = \M_{\p} \cdot x_{\p} $ (with $ x_{\p} $ as defined in theorem \ref{thm_local_generators}) for each prime ideal $ \p $ of $ \OK $. The $ \M $-module $ \M \OL $ therefore defines a class in the locally free class group $ \Cl{\M} $. Since $ H $ is commutative it satisfies the Eichler condition, and so $ \M \OL $ is a free $ \M $-module if and only if its class in $ \Cl{\M} $ is trivial.  The fact that $ H $ is commutative also implies that there is an isomorphism
\[ \Cl{\M} \cong \frac{ \mathbb{J}(H) }{ H^{\times} \mathbb{U}(\M) }, \]
where $ \mathbb{J}(H) $ denotes the id\'{e}le group of $ H $, $ H^{\times} $ denotes the subgroup of principal id\'{e}les, and $ \mathbb{U}(\M) $ is the subgroup of unit id\'{e}les of $ \M $. Using the explicit generators determined in theorem \ref{thm_local_generators}, we can identify the id\'{e}le whose coset corresponds to the class of $ \M\OL $ in $ \Cl{\M} $:

\begin{proposition} \label{prop_class_of_MOL}
The class of $ \M\OL $ in $\Cl{\M} $ corresponds to the coset of the id\'{e}le $ (h_{\p})_{\p} $, where
\[ h_{\p} = \left\{ \begin{array}{ >{\displaystyle}ll} 
1 & \mbox{if $ \p \divides p\OK $} \\
\sum_{j=0}^{p-1} \frac{e_{j}}{ \pi_{\p}^{r_{\p}(a^{j})}} & \mbox{otherwise} \end{array} \right. \]
\end{proposition}
\begin{proof}
Let $ {\displaystyle x = \frac{1}{p} \sum_{j=0}^{p-1} \alpha^{j} } $. Then $ x $ is a free generator of $ L $ as an $ H $-module since by proposition \ref{prop_action_of_idempotents} we have $ e_{i} \cdot x = (1/p) \alpha^{i} $ for each $ i $, and $ \{1,\alpha, \ldots ,\alpha^{p-1} \} $ is a $ K $-basis of $ L $. Therefore the class of $ \OL $ in $ \Cl{\M} $ corresponds to the idele $ (h_{\p})_{\p} $, where for each prime $ \p $ of $ \OK $ the element $ h_{\p} \in H_{\p} $ is defined by $ h_{\p} \cdot x = x_{\p} $ and $ x_{\p} $ is defined as in theorem \ref{thm_local_generators}. The result follows immediately. 
\end{proof}

Since $ H \cong K^{p} $ and $ \M \cong \OK^{p} $, we have isomorphisms 
\[ \frac{ \mathbb{J}(H) }{ H^{\times} \mathbb{U}(\M) } \cong \left( \frac{ \mathbb{J}(K) }{ K^{\times} \mathbb{U}(\OK) } \right)^{p} \cong \Cl{\OK}^{p}, \]
where $ \Cl{\OK} $ is the ideal class group of $ \OK $. Explicitly, if $ (h_{\p})_{\p} \in \mathbb{J}(H) $ then, writing  $ {\displaystyle h_{\p} = \sum_{j=0}^{p-1} z_{j,\p} e_{j}  } $ with $ z_{j,\p} \in K_{\p} $ for each $ \p $, the id\'{e}le $ (h_{\p})_{\p} $ is then mapped to the tuple of classes of ideals
\[ \left( \prod_{\p} \p^{v_{\p}(z_{0,\p})}, \prod_{\p} \p^{v_{\p}(z_{1,\p})}, \ldots , \prod_{\p} \p^{v_{\p}(z_{p-1,\p})} \right). \]

Using this, we obtain a criterion for $ \M\OL $ to be a free $ \M $-module in terms of certain ideals of $ \OK $ being principal. Recall from section \ref{section_classical_version} that the {\em ideals associated} to $ a\OK $ are
\[ \mathfrak{b}_{j} = \prod_{\p} \p^{r_{\p}(a^{j})} \mbox{ for $ 0\leq j \leq p-1 $}, \]
where $ r_{\p}(a^{j}) = \lfloor v_{\p}(a^{j}) / p \rfloor $. 

\begin{proposition} \label{prop_MOL_free}
The $ \M $-module $ \M \OL $ is free if and only if $ \mathfrak{b}_{j} $ is principal for all \linebreak $ j=0,1, \ldots ,p-1 $.
\end{proposition}
\begin{proof}
As discussed above, $ \M \OL $ is a free $ \M $ module if and only if it has trivial class in $ \Cl{\M} $, the class of $ \M \OL $ in $ \Cl{\M} $ corresponds to the class of the id\'{e}le $ (h_{\p})_{\p} $ defined in proposition \ref{prop_class_of_MOL}, and this id\'{e}le corresponds to the tuple of classes of ideals
\[ \left( \prod_{\p \mid a} \p^{-r_{\p}(a^{0})}, \prod_{\p \mid a} \p^{-r_{\p}(a^{1})}, \ldots , \prod_{\p \mid a} \p^{-r_{\p}(a^{p-1})} \right) = \left( \mathfrak{b}_{0}^{-1}, \mathfrak{b}_{1}^{-1}, \ldots ,\mathfrak{b}_{p-1}^{-1} \right). \]
Therefore $ \M \OL $ is a free $ \M $-module if and only if $ \mathfrak{b}_{j} $ is principal for all $ j=0,1, \ldots ,p-1 $.
\end{proof}

Next we turn to the question of when $ \M\OL $ has a free generator lying in $ \OL $:

\begin{proposition} \label{prop_MOL_free_gen_in_OL}
The $ \M $-module $ \M\OL $ has a free generator lying in $ \OL $ if and only if  the ideals $ \mathfrak{b}_{j} $ associated to $ a\OK $ are principal, with generators $ a_{0},a_{1}, \ldots ,a_{p-1} \in \OK $ such that 
\[ \frac{1}{p} \sum_{j=0}^{p-1} \frac{ \alpha^{j} }{a_{j}} \in \OL. \]
\end{proposition}
\begin{proof}
By proposition \ref{prop_MOL_free} $ \M \OL $ is a free $ \M $ module if and only if $ \mathfrak{b}_{j} $ is principal for all $ j=0,1, \ldots ,p-1 $. Suppose that each ideal $ \mathfrak{b}_{j} $ is principal, say $ \mathfrak{b}_{j} = b_{j}\OK $ for each $ j=0,1, \ldots ,p-1 $. Then a free generator of $ \M \OL $ over $ \M $ is given by
\[ x = \frac{1}{p} \sum_{j=0}^{p-1} \frac{\alpha^{j}}{b_{j}}, \]
and the set of free generators for $ \M \OL $ as an $ \M $-module is precisely the set $ \{ z \cdot x \mid z \in \M^{\times} \} $. Recalling that we have $ \M \cong \OK^{p} $ and $ e_{i}\alpha^{j} = \delta_{i,j}\alpha^{j} $ for $ i,j=0,1, \ldots ,p-1 $, we see that an element $ y^{\prime} \in L $ is a free generator for $ \M \OL $ as an $ \M $-module if and only if it has the form
\[ x^{\prime} = \frac{1}{p} \sum_{j=0}^{p-1} \frac{u_{j}\alpha^{j}}{b_{j}} \]
with $ u_{j} \in \OK^{\times} $ for $ j=0,1, \ldots ,p-1 $. Therefore $ \M\OL $ has a free generator lying in $ \OL $  if and only if there are elements $ u_{0},u_{1}, \ldots ,u_{p-1} \in \OK^{\times} $ such that the corresponding element $ x^{\prime} $ lies in $ \OL $. Writing $ a_{j}=u_{j}^{-1}b_{j} $ for each $ j $, this is equivalent to the existence of elements $ a_{j} $ as in the proposition. 
\end{proof}

By combining the results of this section we obtain a criterion for $ \OL $ to be a free $ \A $-module:

\begin{theorem} \label{thm_criterion_free}
The ring of algebraic integers $ \OL $ is a free $ \A $-module if and only if there exists $ \beta \in \OL $ such that
\begin{enumerate}
\item $ L=K(\beta) $;
\item $ b = \beta^{p} \in \OK $;
\item the ideals associated to $ b\OK $ are principal with generators $ b_{j} $ such that 
\[ \sum_{j=0}^{p-1} \frac{\beta^{j}}{b_{j}} \equiv 0 \pmod{p \OL}. \]
\end{enumerate}
Furthermore, in this case the element
\[ \frac{1}{p} \sum_{j=0}^{p-1} \frac{\beta^{j}}{b_{j}}  \]
is a free generator of $ \OL $ as an $ \A $-module. 
\end{theorem}
\begin{proof}
If $ \OL $ is a free $ \A $-module then by the result of Bley and Johnston $ \M\OL = \M \cdot x $ for some $ x \in \OL $. Therefore by proposition \ref{prop_MOL_free_gen_in_OL} the ideals associated to $ a\OK $ are principal, with generators $ a_{0},a_{1}, \ldots ,a_{p-1} \in \OK $ such that 
\[ \frac{1}{p} \sum_{j=0}^{p-1} \frac{ \alpha^{j} }{a_{j}} \in \OL, \]
and so the element $ \beta = \alpha \in \OL $ satisfies (1),(2), and (3). Conversely, suppose that $ \beta \in \OL $ satisfies (1),(2), and (3). We follow the argument of \cite[Remark 1]{CorsoRossi10}. Since $ L=K(\beta) $, we have $ \beta = \alpha^{\ell} c $ for some $ \ell = 1,2, \ldots ,p-1 $ and $ c \in K $. Let $ t $ be the inverse of $ \ell $ modulo $ p $, and for each $ j=0,1, \ldots ,p-1 $ let
\[ a_{j} = b_{\overline{jt}} c^{-jt} a^{- \lfloor ljt / p \rfloor} \in \OL, \]
where $ \overline{jt} $ denotes the principal remainder of $ jt $ modulo $ p $. Then the elements $ a_{j} $ generate that ideals associated to $ a\OK $, and there is an equality of sets 
\[ \{ 1, \beta / b_{1}, \ldots ,\beta^{p-1} / b_{p-1} \} = \{ 1, \alpha / a_{1}, \ldots , \alpha^{p-1} / a_{p-1} \}. \]
Therefore 
\[ \sum_{j=0}^{p-1} \frac{\alpha^{j}}{a_{j}} = \sum_{j=0}^{p-1} \frac{\beta^{j}}{b_{j}} \equiv 0 \pmod{p \OL}, \]
so by proposition \ref{prop_MOL_free_gen_in_OL} we have $ \M\OL = \M \cdot x $ for some $ x \in \OL $, and so by the result of Bley and Johnston $ \OL $ is a free $ \A $-module. 
\end{proof}

Finally, we observe that the conditions appearing in this criterion are identical to those appearing in G\'{o}mez Ayala's criterion, as summarized in section \ref{section_classical_version}.

\section{A uniform approach to the normal and non-normal cases} \label{section_uniform}

We have seen that a radical extension of number fields $ L/K $ of degree $ p $ admits a unique Hopf-Galois structure: that given by $ H= K[G] $ (with $ G=\Gal{L/K} $) if $ L/K $ is normal, and that given by $ H=E[N]^{G} $ (with $ E $ the Galois closure of $ L/K $, $ G =\Gal{E/K}$, and $ N $ as in proposition \ref{prop_determining_regular_subgroup}) if $ L/K $ is non-normal. In either case we have $ H \cong K^{p} $ as $ K $-algebras: in the normal case because $ K $ contains a primitive $ p\textsuperscript{th} $ root of unity; in the non-normal case by proposition \ref{prop_K_algebra_isomorphism}. Writing $ L=K(\alpha) $ for some $ \alpha \in L $ such that $ \alpha^{p} \in K $ and renumbering if necessary, the orthogonal idempotents in $ H $ then act by $ e_{i} \cdot \alpha^{j} = \delta_{i,j}\alpha^{j} $ for $ i,j=0,1, \ldots ,p-1 $. This uniformity in the $ K $-algebra structure of $ H $ and its action on $ L $ has implications for questions of integral module structure. 

Suppose that $ L/K $ is tame. Assuming that $ p $ is unramified in $ K $ in the non-normal case, we can then choose $ \alpha $ such that  $ \alpha^{p} \equiv 1 \pmod{(\zeta-1)^{p}\OK} $ if $ L/K $ is normal and $ \alpha^{p} \equiv 1 \pmod{p^{2}\OK} $ if $ L/K $ is non-normal (proposition \ref{prop_criterion_for_tame}). In either case,  the ring of algebraic integers $ \OL $ is locally free over its associated order in $ H $: in the normal case by Noether's theorem and in the non-normal case by theorem \ref{thm_local_generators}. Moreover, the local generators of $ \OLp $ over $ \Ap $ are the same in both cases: for a prime ideal $ \p $ of $ \OK $ not lying above $ p $ the orthogonal idempotents in $ H_{\p} $ form an $ \OKp $-basis of $ \Ap $ in both cases, so the appropriate parts of the proof of theorem \ref{thm_local_generators} apply equally well to the normal and non-normal case. For $ \p $ lying above $ p $ a small modification to the argument of theorem \ref{thm_local_generators} shows that the element $ x_{\p} $ defined there is also a free generator of $ \OLp $ over $ \Ap $ in the normal case. (Alternatively, since in the normal case $ \Ap $ is a local Hopf order for such $ \p $, we could deduce this from the Childs-Hurley criterion: see \cite[Theorem 14.7]{Childs00}.) 

Given that $ \OL $ is a locally free $ \A $-module in both case, the result of Bley and Johnston \cite[Proposition 2.1]{BleyJohnston08} implies that it is a free $ \A $-module if and only if $ \M\OL $ is a free $ \M $-module with a generator lying in $ \OL $ (where $ \M $ denotes the unique maximal order in $ H $). But we have $ \M \cong \OK^{p} $ via orthogonal idempotents in both cases, and for each $ \p $ the element $ x_{\p} $ defined in theorem \ref{thm_local_generators} is a free generator of $ (\M\OL)_{\p} $ as an $ \M_{\p} $-module. Therefore all of the arguments in section \ref{section_HGMS} involving the id\'{e}lic description of $ \Cl{\M} $ and the id\'{e}le corresponding to the class of $ \M\OL $ apply equally well in both cases, which explains why the criterion we obtained in theorem \ref{thm_criterion_free} is identical to that obtained by G\'{o}mez Ayala.

\end{document}